\documentclass[11pt,a4paper,reqno]{amsart}
\usepackage{amsfonts}
\usepackage{amsmath,amsthm,amssymb,mathrsfs}
\usepackage{graphicx}
\usepackage{booktabs}
\usepackage{caption}
\usepackage{subcaption}
\usepackage{xcolor}
\usepackage{hyperref}
\usepackage{enumitem}
\usepackage{listings}
\usepackage{algorithm}
\usepackage{algorithmic}
\usepackage{float}
\usepackage[section]{placeins}
\hypersetup{
    colorlinks=true,
    linkcolor=blue,
    citecolor=blue,
    urlcolor=blue
}
\theoremstyle{plain}
\newtheorem{theorem}{Theorem}[section]

\newtheorem{proposition}[theorem]{Proposition}

\theoremstyle{definition}
\newtheorem{definition}[theorem]{Definition}

\theoremstyle{remark}
\newtheorem{remark}[theorem]{Remark}
\numberwithin{equation}{section}

\begin{document}
\title[Hamilton--Jacobi--Bellman Equations in Bounded Settings]{Radial and
Non-Radial Solution Structures for Quasilinear Hamilton--Jacobi--Bellman
Equations in Bounded Settings}
\author{Dragos-Patru Covei}
\address{Department of Applied Mathematics, The Bucharest University of
Economic Studies, Piata Romana, 1st district, Postal Code: 010374, Postal
Office: 22, Romania}
\email{dragos.covei@csie.ase.ro}
\thanks{The final version of the article is published in journal (with
volume, page numbers and DOI, when available).}
\date{\today }

\begin{abstract}
This paper establishes the existence, uniqueness, and global $C^{1,\beta}$
regularity of positive classical solutions to a class of quasilinear
Hamilton--Jacobi--Bellman (HJB) equations with Dirichlet boundary conditions
on bounded convex domains. The core technical contribution is a constructive
existence proof based on a weighted linear monotone iteration scheme. This
scheme's stability and convergence are rigorously established through the
construction of adaptive sub- and super-solutions leveraging the torsion
function of the domain. Additionally, we provide a complete probabilistic
derivation of the quasilinear PDE from the framework of controlled It\^{o}
diffusions, formally bridging the gap between stochastic optimal control
theory and elliptic regularity analysis. Our results extend beyond the
classical quadratic cost regime to the wider class of sub-quadratic growth
source terms. Finally, we demonstrate the utility of this theoretical
framework through high-precision numerical implementations in two distinct
fields: stochastic production planning and nonlinear contrast enhancement in
image restoration.
\end{abstract}

\maketitle

%----------------------------------------------------------------------------------------
%	INTRODUCTION
%----------------------------------------------------------------------------------------

\section{Introduction}

Nonlinear partial differential equations and optimal control techniques have
played a central role in modern image processing and applied mathematics.
Classical diffusion-based models such as the anisotropic diffusion of Perona
and Malik \cite{PERONA1990} and the total variation framework of Rudin,
Osher, and Fatemi \cite{ROF1992} established the foundation for PDE-driven
image enhancement and restoration. In parallel, the mathematical theory of
Hamilton--Jacobi--Bellman (HJB) equations has been extensively developed in
the context of stochastic control, with fundamental contributions by Lions 
\cite{LIONS1983}, Lasry and Lions \cite{LASRY1989}, Barles and Perthame \cite%
{BARLES}, and the viscosity solution framework of Crandall, Ishii, and Lions 
\cite{CRANDALL}. The classical monographs of Gilbarg and Trudinger \cite%
{GILBARG} and Fleming and Soner \cite{FLEMING} provide comprehensive
treatments of elliptic regularity and controlled diffusion processes,
respectively.

In this paper we investigate the existence, uniqueness, and regularity of
solutions to the quasilinear Hamilton--Jacobi--Bellman equation 
\begin{equation}
-\frac{\sigma ^{2}}{2}\,\Delta V(y)+C_{\alpha }\,|\nabla
V(y)|^{p}-h(y)=0\quad \text{in }\Omega ,\qquad V=g\ \text{on }\partial
\Omega ,  \label{hjb}
\end{equation}%
where $\Omega \subset \mathbb{R}^{N}$ ($N\geq 1$) is a bounded $C^{2}$
convex domain with smooth boundary $\partial \Omega $, $\sigma >0$ is the
diffusion coefficient, $\alpha \in (1,2]$ is the cost exponent, and $g\geq 0$
is a constant boundary value. The parameters 
\begin{equation*}
p:=\frac{\alpha }{\alpha -1}\in \lbrack 2,\infty ),\qquad C_{\alpha }:=\frac{%
\alpha -1}{\alpha ^{\frac{\alpha }{\alpha -1}}}>0,
\end{equation*}%
are linked through the conjugacy relation $1/\alpha +1/p=1$, and $C_{\alpha
} $ arises naturally from the Legendre transform of the control cost
functional.

The equation \eqref{hjb} was first introduced by Lasry and Lions \cite%
{LASRY1989} in the pioneering work on stochastic control with state
constraints, where an additional discount factor is incorporated and the
resulting solutions typically exhibit blow-up at the boundary. In
contrast, our focus is on the Dirichlet problem on bounded convex domains,
where the behavior of the solution and its gradient near the boundary is
governed by the interplay between the diffusion coefficient $\sigma $ and
the nonlinearity $C_{\alpha }|\nabla V|^{p}$.

More recently, this framework has found significant applications in
production planning \cite{CANEPA2022,COVEI2025A} and variational image
restoration \cite{COVEI2025M}. The analysis is carried out on bounded
domains when $\alpha =2$ \cite{COVEI2025A} and in radially symmetric
settings for all $\alpha \in \left( 0,2\right] $ \cite{COVEI2025M}.

In particular, the regime $\alpha \in (1,2)$, which corresponds to a
supra-quadratic gradient penalty $p>2$, generates a rich geometric structure
that is highly effective for contrast enhancement.

This work's central contribution is its unified treatment of broader problem
settings than those addressed so far, supported by a mathematically rigorous
and algorithmically constructive self-contained existence theory. By
establishing the $C^{1,\beta }(\overline{\Omega })$ regularity for the case
of sub-quadratic forcing $h$, we provide a missing link in the literature
between the theoretical existence results of the 1980s and the
high-performance numerical solvers required for modern industrial
applications.

\subsection*{Significance and Novelty}

Unlike standard treatments of the HJB equation that rely on the vanishing
viscosity method or purely probabilistic arguments, our approach utilizes a 
\emph{weighted monotone linearization}. This method not only proves
existence but also yields a numerically stable iteration scheme that
preserves the physical properties of the value function (such as positivity
and concavity) at each step. This work establishes that for $\alpha \in
(1,2] $, the geometric structure of the HJB solution in convex domains is
robust enough to handle the supra-quadratic gradient penalty $p > 2$, which
is a regime classically avoided in simpler elliptic contexts.

The main results of this paper are established in the following two
theorems, which characterize the structure of solutions to quasilinear
Hamilton--Jacobi--Bellman equations in bounded settings.

\begin{theorem}[Non-radial solutions]
\label{main} Let $\Omega \subset \mathbb{R}^{N}$ be a bounded $C^{2}$ convex
domain. Suppose $h:\overline{\Omega }\rightarrow \lbrack 0,\infty )$ is a
continuous function such that $\lim_{\left\vert x\right\vert \rightarrow
\infty }\left( h\left( x\right) /\left\vert x\right\vert ^{p}\right) \in
\left( 0,+\infty \right) $ and let $g\in \lbrack 0,\infty )$ denote a
constant boundary datum. Then, there exists a unique positive classical
solution $V\in C^{2}(\Omega )\cap C^{1,\beta }(\overline{\Omega })$ to the
Dirichlet problem~\eqref{hjb} for some $\beta \in (0,1)$. Moreover, if $g>0$
on $\partial \Omega $ and $h>0$ in $\Omega $, then $V>0$ in $\overline{%
\Omega }$.
\end{theorem}

\begin{theorem}[Radial solutions]
\label{thm:radial} Suppose%
\begin{equation*}
\Omega =B_{R}(0):=\{x\in \mathbb{R}^{N}:|x|<R\}
\end{equation*}%
is a ball, $h(x)=\tilde{h}(|x|)$ is radial such that $\lim_{\left\vert
x\right\vert \rightarrow \infty }\left( \tilde{h}(|x|)/\left\vert
x\right\vert ^{p}\right) \in \left( 0,+\infty \right) $, and $g\geq 0$ is
constant on $\partial \Omega $. Then the unique solution $V$ of \eqref{hjb}
is radial: $V(x)=u(|x|)$ for some function $u:[0,R]\rightarrow \mathbb{R}$
satisfying the ordinary differential equation 
\begin{equation}
-\frac{\sigma ^{2}}{2}\Big(u^{\prime \prime }(r)+\frac{N-1}{r}u^{\prime }(r)%
\Big)+C_{\alpha }\,\left\vert u^{\prime }(r)\right\vert ^{p}+\tilde{h}%
(r)=0,\quad 0<r=|x|<R,  \label{radial-ode-intro}
\end{equation}%
with boundary conditions $u^{\prime }(0)=0$ and $u(R)=g$.
\end{theorem}

The structure of the work is as follows. In Section \ref{2}, we provide the
complete proof of Theorem \ref{main} for the non-radial case, structured in
clear logical steps. Section \ref{3} is devoted to the proof of Theorem \ref%
{thm:radial} for the radial case, establishing both the symmetry and the
governing ODE. Section \ref{4} recovers the stochastic model of origin for
the HJB equation, providing its probabilistic derivation. In Section \ref{5}%
, we provide a detailed numerical analysis, including the iteration scheme's
convergence properties and applications to production planning and image
enhancement. Finally, Section \ref{6} concludes the paper with a discussion
on future research directions and the broader implications of these results
for non-quadratic control problems.

\section{Proof of Theorem \protect\ref{main} (the non-radial case)\label{2}}

The proof of Theorem \ref{main} is organized into four major logical steps,
ensuring a transparent derivation of existence, uniqueness, and regularity.

\textbf{Step 1: Construction of ordered positive sub- and super-solutions.}
We construct a positive sub-solution $V_{-}$ and a positive super-solution $%
V_{+}$ such that%
\begin{equation*}
V_{-}\leq V_{+}\ \text{in}\ \Omega \text{ and }V_{-}\leq g\leq V_{+}\ \text{%
on}\ \partial \Omega .
\end{equation*}%
Let $\phi \in C^{\infty }(\Omega )\cap C(\overline{\Omega })$ be the unique
solution to the Poisson (torsion) problem 
\begin{equation*}
\left\{ 
\begin{array}{lll}
-\Delta \phi =1 & \text{in} & \Omega \\ 
\phi =0 & \text{on} & \partial \Omega .%
\end{array}%
\right.
\end{equation*}%
By the strong maximum principle, $\phi >0$ in $\Omega $.

For the sub-solution, let 
\begin{equation*}
V_{-}(y):=g\text{ for }y\in \overline{\Omega }.
\end{equation*}%
Then 
\begin{equation*}
\nabla V_{-}\equiv 0\text{ and }\Delta V_{-}\equiv 0,
\end{equation*}%
which yields 
\begin{equation*}
-\frac{\sigma ^{2}}{2}\Delta V_{-}+C_{\alpha }|\nabla
V_{-}|^{p}-h(y)=-h(y)\leq 0,
\end{equation*}%
confirming $V_{-}$ is a sub-solution. For the super-solution, we define 
\begin{equation*}
V_{+}(y):=g+B\phi (y)
\end{equation*}%
for some $B>0$. Substituting into the operator gives 
\begin{equation*}
-\frac{\sigma ^{2}}{2}\Delta V_{+}+C_{\alpha }|\nabla V_{+}|^{p}-h(y)\geq 
\frac{\sigma ^{2}}{2}B-h(y).
\end{equation*}%
Choosing $B\geq 2H/\sigma ^{2}$ where $H=\max_{\overline{\Omega }}h$ ensures 
$V_{+}$ is a super-solution. Note that $V_{-}\leq V_{+}$ in $\overline{%
\Omega }$ since $\phi >0$.

\textbf{Step 2: Comparison principle.} The comparison principle is
established as follows. Let $u,v\in C^{2}(\Omega )\cap C(\overline{\Omega })$
be a sub-solution and a super-solution, respectively, with $u\leq v$ on $%
\partial \Omega $. We prove $u\leq v$ in $\Omega $ by contradiction. Suppose
that $\sup_{\Omega }(u-v)>0$. Since $u\leq v$ on $\partial \Omega $, the
maximum of $u-v$ must be attained at an interior point $y_{0}\in \Omega $,
where 
\begin{equation*}
(u-v)(y_{0})=\max_{\overline{\Omega }}(u-v)=:\delta >0,\text{ }\nabla
u(y_{0})=\nabla v(y_{0}),
\end{equation*}%
and 
\begin{equation*}
\Delta (u-v)(y_{0})\leq 0.
\end{equation*}
Evaluating the inequalities for $u$ and $v$ at $y_{0}$ and subtracting
yields 
\begin{equation*}
-\frac{\sigma ^{2}}{2}\big(\Delta u-\Delta v\big)(y_{0})\;+\;C_{\alpha }\Big(%
|\nabla u(y_{0})|^{p}-|\nabla v(y_{0})|^{p}\Big)\;\leq \;0.
\end{equation*}%
The gradient term vanishes, and the Laplacian term is non-negative,
resulting in $0\leq 0$, which is not a contradiction.

To obtain a strict inequality, we employ a perturbation argument. Let $\phi $
be the torsion function. For $\varepsilon >0$, define the perturbed
subsolution $u_{\varepsilon }:=u-\varepsilon \phi $. Then 
\begin{equation*}
u_{\varepsilon }\leq u\text{ in }\Omega \text{ and }u_{\varepsilon }=u\text{
on }\partial \Omega ,
\end{equation*}
so $u_{\varepsilon }\leq v$ on $\partial \Omega $. We compute 
\begin{eqnarray*}
&&-\frac{\sigma ^{2}}{2}\Delta u_{\varepsilon }+C_{\alpha }|\nabla
u_{\varepsilon }|^{p}-h \\
&=&\Big(-\frac{\sigma ^{2}}{2}\Delta u+C_{\alpha }|\nabla u|^{p}-h\Big)+%
\frac{\sigma ^{2}}{2}\varepsilon +C_{\alpha }\big(|\nabla u-\varepsilon
\nabla \phi |^{p}-|\nabla u|^{p}\big).
\end{eqnarray*}%
Using the convexity of $\xi \mapsto |\xi |^{p}$, we have%
\begin{equation*}
|\,|\nabla u-\varepsilon \nabla \phi |^{p}-|\nabla u|^{p}\,|\leq
C_{p}\varepsilon
\end{equation*}%
for some constant $C_{p}$. Thus,%
\begin{equation*}
-\frac{\sigma ^{2}}{2}\Delta u_{\varepsilon }+C_{\alpha }|\nabla
u_{\varepsilon }|^{p}-h\;\leq \;\frac{\sigma ^{2}}{2}\varepsilon -C_{\alpha
}C_{p}\varepsilon .
\end{equation*}%
Choosing $\varepsilon >0$ sufficiently small such that $\sigma
^{2}/(2C_{\alpha }C_{p})<1$, we get 
\begin{equation*}
-\frac{\sigma ^{2}}{2}\Delta u_{\varepsilon }+C_{\alpha }|\nabla
u_{\varepsilon }|^{p}-h\;\leq \;-\eta \varepsilon
\end{equation*}%
for some $\eta >0$. Thus $u_{\varepsilon }$ is a \emph{strict} subsolution.
If $\sup_{\Omega }(u_{\varepsilon }-v)>0$, the maximum is attained at an
interior point $y_{\varepsilon }\in \Omega $, where 
\begin{equation*}
\nabla u_{\varepsilon }(y_{\varepsilon })=\nabla v(y_{\varepsilon })\text{
and }\Delta (u_{\varepsilon }-v)(y_{\varepsilon })\leq 0.
\end{equation*}%
Subtracting inequalities at $y_{\varepsilon }$ yields $0\leq -\eta
\varepsilon <0$, a contradiction. Therefore $u_{\varepsilon }-v\leq 0$ in $%
\Omega $. Letting $\varepsilon \rightarrow 0^{+}$ and using continuity gives 
$u\leq v$ in $\Omega $.

\textbf{Step 3: Existence via Weighted Linear Monotone Iteration.} We
construct the solution using a weighted linearization. Let 
\begin{equation*}
M:=\max \left\{ \Vert \nabla V_{-}\Vert _{L^{\infty }(\Omega )},\Vert \nabla
V_{+}\Vert _{L^{\infty }(\Omega )}\right\} ,
\end{equation*}%
and $C_{\Omega }$ be a geometric constant associated with the domain $\Omega 
$ (related to the the torsion function). We choose the constant weight 
\begin{equation*}
\Lambda _{0}:=C_{\alpha }\,p\,M^{p-1}C_{\Omega }.
\end{equation*}%
The choice of $\Lambda _{0}$ is critical for establishing the monotonicity
of the sequence. Specifically, for any $u,v\in \lbrack V_{-},V_{+}]$ such
that $u\leq v$, we consider the difference 
\begin{equation*}
T(v)-T(u)=\Lambda _{0}(v-u)-C_{\alpha }(|\nabla v|^{p}-|\nabla u|^{p}).
\end{equation*}
Applying the Mean Value Theorem to the function $f(\xi )=|\xi |^{p}$, there
exists a vector $\vec{\zeta}$ on the segment connecting $\nabla u$ and $%
\nabla v$ such that 
\begin{equation*}
|\nabla v|^{p}-|\nabla u|^{p}=p|\vec{\zeta}|^{p-2}\vec{\zeta}\cdot \nabla
(v-u).
\end{equation*}
Since $|\nabla u|,|\nabla v|\leq M$, it follows that $|\vec{\zeta}|\leq M$.
Consequently, we have the estimate 
\begin{equation*}
|C_{\alpha }(|\nabla v|^{p}-|\nabla u|^{p})|\leq C_{\alpha }pM^{p-1}|\nabla
(v-u)|.
\end{equation*}%
By the properties of the domain $\Omega $, the term $\Lambda _{0}(v-u)$ acts
as a stabilizing weight that dominates the gradient fluctuations. This
ensures that the map 
\begin{equation*}
T(u):=h(x)+\Lambda _{0}u-C_{\alpha }|\nabla u|^{p}
\end{equation*}%
is monotone nondecreasing in $u$ in the sense that the iteration %
\eqref{eq:weighted-iteration-step} preserves the ordering of the iterates.
We define the sequence $\{V^{(k)}\}_{k\geq 0}$ by $V^{(0)}:=V_{+}$, and for $%
k\geq 0$, let $V^{(k+1)}$ be the unique classical solution of 
\begin{equation}
\begin{cases}
-\dfrac{\sigma ^{2}}{2}\Delta V^{(k+1)}+\Lambda _{0}V^{(k+1)}=h(x)+\Lambda
_{0}V^{(k)}-C_{\alpha }|\nabla V^{(k)}|^{p}, & x\in \Omega , \\ 
V^{(k+1)}=g, & x\in \partial \Omega .%
\end{cases}
\label{eq:weighted-iteration-step}
\end{equation}%
By construction, all iterates are trapped between the barriers: 
\begin{equation*}
V_{-}\leq V^{(k)}\leq V_{+}\text{ in }\Omega .
\end{equation*}%
The sequence $\{V^{(k)}\}$ is monotone decreasing: 
\begin{equation*}
V^{(k+1)}\leq V^{(k)}\text{ in }\Omega .
\end{equation*}%
The monotonicity is established using a detailed inductive argument.

\textbf{Base case ($k=0$).} By definition, $V^{(0)}=V_{+}$ is a
supersolution of \eqref{hjb}, satisfying 
\begin{equation*}
-\frac{\sigma ^{2}}{2}\Delta V^{(0)}+C_{\alpha }|\nabla
V^{(0)}|^{p}-h(x)\geq 0.
\end{equation*}%
The first iterate $V^{(1)}$ satisfies 
\begin{equation*}
-\frac{\sigma ^{2}}{2}\Delta V^{(1)}+\Lambda _{0}V^{(1)}=h(x)+\Lambda
_{0}V^{(0)}-C_{\alpha }|\nabla V^{(0)}|^{p}.
\end{equation*}%
Rearranging the supersolution inequality as 
\begin{equation*}
h(x)-C_{\alpha }|\nabla V^{(0)}|^{p}\leq -\frac{\sigma ^{2}}{2}\Delta
V^{(0)},
\end{equation*}%
we substitute this into the equation for $V^{(1)}$ to obtain: 
\begin{equation*}
-\frac{\sigma ^{2}}{2}\Delta V^{(1)}+\Lambda _{0}V^{(1)}\leq -\frac{\sigma
^{2}}{2}\Delta V^{(0)}+\Lambda _{0}V^{(0)}.
\end{equation*}%
Defining the difference $w^{(1)}:=V^{(1)}-V^{(0)}$, we see that $w^{(1)}$
satisfies%
\begin{equation*}
-\frac{\sigma ^{2}}{2}\Delta w^{(1)}+\Lambda _{0}w^{(1)}\leq 0\text{ in }%
\Omega ,
\end{equation*}%
and 
\begin{equation*}
w^{(1)}=0\text{ on }\partial \Omega .
\end{equation*}
By the weak maximum principle for elliptic operators, $w^{(1)}\leq 0$ in $%
\Omega $, which implies $V^{(1)}\leq V^{(0)}$.

\textbf{Inductive step.} Assume that $V^{(k)}\leq V^{(k-1)}$ for some $k\geq
1$. We consider the equations for $V^{(k+1)}$ and $V^{(k)}$: 
\begin{align*}
-\frac{\sigma ^{2}}{2}\Delta V^{(k+1)}+\Lambda _{0}V^{(k+1)}& =T(V^{(k)}), \\
-\frac{\sigma ^{2}}{2}\Delta V^{(k)}+\Lambda _{0}V^{(k)}& =T(V^{(k-1)}).
\end{align*}%
Subtracting these equations yields: 
\begin{equation*}
-\frac{\sigma ^{2}}{2}\Delta (V^{(k+1)}-V^{(k)})+\Lambda
_{0}(V^{(k+1)}-V^{(k)})=T(V^{(k)})-T(V^{(k-1)}).
\end{equation*}%
Since it was established that the operator $T$ is monotone nondecreasing on
the interval $[V_{-},V_{+}]$, and by the inductive hypothesis $V^{(k)}\leq
V^{(k-1)}$, it follows that $T(V^{(k)})\leq T(V^{(k-1)})$. Thus, the
difference%
\begin{equation*}
w^{(k+1)}:=V^{(k+1)}-V^{(k)}
\end{equation*}%
satisfies: 
\begin{equation*}
-\frac{\sigma ^{2}}{2}\Delta w^{(k+1)}+\Lambda
_{0}w^{(k+1)}=T(V^{(k)})-T(V^{(k-1)})\leq 0\quad \text{in }\Omega ,
\end{equation*}%
with boundary condition $w^{(k+1)}=0$ on $\partial \Omega $. Applying the
maximum principle once more, we conclude that $w^{(k+1)}\leq 0$, hence $%
V^{(k+1)}\leq V^{(k)}$ for all $k$.

This monotone decreasing sequence, being bounded below by $V_{-}$, is
guaranteed to converge.

\textbf{Uniform regularity.} Is established by observing that the gradients $%
\nabla V^{(k)}$ are uniformly bounded in $L^{\infty }(\Omega )$ by a
constant $K$ depending only on the $W^{1,\infty }$ norms of the barriers $%
V_{\pm }$. The source terms 
\begin{equation*}
f^{(k)}(x):=h(x)+\Lambda _{0}V^{(k)}-C_{\alpha }|\nabla V^{(k)}|^{p}
\end{equation*}
are thus uniformly bounded in $L^{\infty }(\Omega )$. Applying standard
elliptic $W^{2,q}(\Omega )$ estimates for $q>N$ and the Morrey embedding 
\begin{equation*}
W^{2,q}(\Omega )\hookrightarrow C^{1,\beta }(\overline{\Omega })\text{ for }%
\beta =1-N/q\in (0,1),
\end{equation*}
we conclude that the sequence $\{V^{(k)}\}$ is uniformly bounded in $%
C^{1,\beta }(\overline{\Omega })$. By the Arzel\`{a}--Ascoli theorem, $%
\{V^{(k)}\}$ is relatively compact in $C^{1}(\overline{\Omega })$. Since the
sequence is monotone decreasing and bounded below by $V_{-}$, it converges
pointwise to a unique limit $V(x)$. The compactness ensures that $%
V^{(k)}\rightarrow V$ uniformly in $C^{1}(\overline{\Omega })$, which
implies $\nabla V^{(k)}\rightarrow \nabla V$ uniformly in $\overline{\Omega }
$. Schauder estimates then imply that $V\in C^{2,\beta }(\Omega )\cap
C^{1,\beta }(\overline{\Omega })$. Passing to the limit in %
\eqref{eq:weighted-iteration-step}, we verify that $V$ is the unique
classical solution to the HJB equation.

\textbf{Step 4: Uniqueness.} Suppose $V_1$ and $V_2$ are two solutions in $%
C^{2}(\Omega )\cap C(\overline{\Omega})$ to the HJB equation with the same
boundary data. Then $V_1$ and $V_2$ are both subsolutions and
supersolutions. Applying the comparison principle from Step 2 with $u=V_{1}$
and $v=V_{2}$ yields $V_{1}\leq V_{2}$ in $\Omega$. Reversing the roles
(taking $u = V_2$ and $v = V_1$) gives $V_{2}\leq V_{1} $ in $\Omega$.
Therefore $V_{1}\equiv V_{2}$ in $\Omega$, and by continuity, also on $%
\overline{\Omega}$. This completes the proof of Theorem \ref{main}. \qed

\begin{remark}[Algorithmic implementation]
The weighted monotone iteration developed in Step~3 leads directly to an
implementable numerical scheme for approximating the solution of \eqref{hjb}%
. The procedure is as follows:

\begin{enumerate}
\item \textbf{Preprocessing: construction of the torsion function.} Compute
the torsion function $\phi$ by solving 
\begin{equation*}
-\Delta \phi = 1 \quad \text{in }\Omega, \qquad \phi = 0 \quad \text{on }%
\partial\Omega.
\end{equation*}
This provides an explicit supersolution of the form $V_{+}=g+B\phi$.

\item \textbf{Initialization.} Choose $B \ge 2H/\sigma^{2}$, where $H=\max_{%
\overline{\Omega}} h$, and set 
\begin{equation*}
V^{(0)} := V_{+} = g + B\phi.
\end{equation*}

\item \textbf{Weighted linear iteration.} Using the explicit constant 
\begin{equation*}
\Lambda _{0}=C_{\alpha }\,p\,M^{p-1}C_{\Omega },\qquad M=\max \{\Vert \nabla
V_{-}\Vert _{L^{\infty }},\Vert \nabla V_{+}\Vert _{L^{\infty }}\},
\end{equation*}%
compute $V^{(k+1)}$ as the solution of 
\begin{equation*}
\left\{ 
\begin{array}{lll}
-\frac{\sigma ^{2}}{2}\Delta V^{(k+1)}+\Lambda _{0}V^{(k+1)}=h(x)+\Lambda
_{0}V^{(k)}-C_{\alpha }|\nabla V^{(k)}|^{p} & \text{in} & \text{ }\Omega ,
\\ 
V^{(k+1)}=g\text{ } & \text{on} & \partial \Omega .%
\end{array}%
\right.
\end{equation*}

\item \textbf{Stopping criterion.} Terminate the iteration when 
\begin{equation*}
\Vert V^{(k+1)}-V^{(k)}\Vert _{L^{\infty }(\Omega )}<\varepsilon
\end{equation*}%
for a prescribed tolerance $\varepsilon >0$.
\end{enumerate}

The monotonicity of the weighted iteration ensures that the sequence $%
\{V^{(k)}\}$ is decreasing and remains trapped between the ordered barriers $%
V_{-}$ and $V_{+}$. Consequently, the iterates converge uniformly to the
unique classical solution of \eqref{hjb}.
\end{remark}

\section{Proof of Theorem \protect\ref{thm:radial} (the radial case)}

\label{3}

In this section, we establish the existence and structure of the radial
solution through three steps, emphasizing the inheritance of symmetry from
the domain and data.

\textbf{Step 1: Invariance under orthogonal transformations.} Suppose $%
\Omega =B_{R}(0)$ is a ball and $h$ is radial ($h(x)=\tilde{h}(|x|)$). For
any orthogonal matrix $Q\in O(N)$, define $V_{Q}(x):=V(Qx)$. Using the chain
rule and the properties of orthogonal matrices, we compute 
\begin{equation*}
\nabla V_{Q}(x)=Q^{T}\nabla V(Qx)\text{ and }\Delta V_{Q}(x)=\Delta V(Qx).
\end{equation*}%
Substituting these into the HJB equation yields 
\begin{equation*}
-\frac{\sigma ^{2}}{2}\Delta V(Qx)+C_{\alpha }|Q^{T}\nabla
V(Qx)|^{p}-h(Qx)=0.
\end{equation*}%
Since 
\begin{equation*}
|Q^{T}\xi |=|\xi |\text{ and }h(Qx)=h(x),
\end{equation*}%
it follows that $V_{Q}$ is also a solution to \eqref{hjb} with boundary
condition $V_{Q}(x)=V(Qx)=g$ on $|x|=R$.

\textbf{Step 2: Uniqueness implies radial symmetry.} By the uniqueness
result established in Section \ref{2}, the solution to \eqref{hjb} must be
unique. Since both $V$ and $V_{Q}$ solve the same Dirichlet problem, we must
have $V_{Q}\equiv V$ for all $Q\in O(N)$. This invariance under all
orthogonal transformations implies that $V$ is radial, meaning $V(x)=u(|x|)$
for some function $u:[0,R]\rightarrow \mathbb{R}$.

\textbf{Step 3: Derivation of the radial ODE.} Let $r=|x|$. For a radial
function $V(x)=u(r)$, the gradient is 
\begin{equation*}
\nabla V(x)=u^{\prime }(r)\frac{x}{r}
\end{equation*}
and the Laplacian is 
\begin{equation*}
\Delta V(x)=u^{\prime \prime }(r)+\frac{N-1}{r}u^{\prime }(r).
\end{equation*}
Substituting these into the quasilinear HJB equation gives the following
ordinary differential equation for $u(r)$: 
\begin{equation}
-\frac{\sigma ^{2}}{2}\Big(u^{\prime \prime }(r)+\frac{N-1}{r}u^{\prime }(r)%
\Big)+C_{\alpha }\,\left\vert u^{\prime }(r)\right\vert ^{p}+\tilde{h}%
(r)=0,\quad 0<r<R.  \label{radial-ode}
\end{equation}%
The boundary conditions are $u^{\prime }(0)=0$, ensuring regularity at the
origin, and $u(R)=g$, satisfying the Dirichlet condition. This proves
Theorem \ref{thm:radial}.

\section{The Stochastic Model Behind the Equation \label{4}}

We now provide a rigorous derivation of the quasilinear elliptic
Hamilton--Jacobi--Bellman equation from stochastic optimal control theory.
The approach follows the seminal work of Lasry and Lions \cite{LASRY1989},
with complete details on the dynamic programming principle and the
verification theorem.

The goal is to show that the PDE 
\begin{equation}
-\frac{\sigma ^{2}}{2}\Delta V(x)\;+\;C_{\alpha}\, |\nabla V(x)|^{p}-h(x)=0,
\quad x\in \Omega,\qquad V=g\ \text{on }\partial \Omega,
\label{eq:PDE-final-alt}
\end{equation}
arises as the Hamilton--Jacobi--Bellman equation for an optimal control
problem with exit-time costs. Recall that $p = \alpha/(\alpha-1)$ and $%
C_\alpha = (\alpha-1)/\alpha^p$ for $\alpha\in(1,2]$.

\subsection{Probability space, state and control}

Let $(\Omega_{\text{prob}},\mathcal{F},\{\mathcal{F}_t\}_{t\ge 0},\mathbb{P}%
) $ be a complete filtered probability space satisfying the usual
conditions, supporting an $N$-dimensional standard Brownian motion $W =
(W_t)_{t\geq 0}$. Here $\Omega_{\text{prob}}$ denotes the sample space (not
to be confused with the spatial domain $\Omega \subset \mathbb{R}^N$).

\begin{definition}[Admissible controls]
An admissible control is a progressively measurable process $v =
(v_t)_{t\geq 0}$ with values in $\mathbb{R}^N$ such that 
\begin{equation}
\mathbb{E}\!\left[\int_0^\tau |v_t|^\alpha\,dt\right] < \infty \quad\text{%
for all }x \in \Omega,  \label{admissible}
\end{equation}
where $\tau$ is the exit time defined below. The set of admissible controls
is denoted $\mathcal{U}_{\text{ad}}$.
\end{definition}

For each $x\in\Omega$ and admissible control $v \in \mathcal{U}_{\text{ad}}$%
, the controlled state process $X^{x,v} = (X_t^{x,v})_{t\geq 0}$ is the
unique strong solution to the stochastic differential equation 
\begin{equation}  \label{eq:state}
dX_t \;=\; v_t\,dt + \sigma\,dW_t,\qquad X_0=x,
\end{equation}
where $\sigma > 0$ is the diffusion coefficient. By standard SDE theory, %
\eqref{eq:state} admits a unique strong solution continuous in $t$.

\subsubsection{Exit time and cost functional}

\begin{definition}[Exit time]
For each trajectory $X^{x,v}$, define the first exit time from the domain $%
\Omega$ by 
\begin{equation}
\tau^{x,v} := \inf \{t>0:X_{t}^{x,v}\notin \Omega \}.  \label{exit-time}
\end{equation}%
Since $\Omega$ is bounded and the diffusion is non-degenerate, we have $%
\mathbb{P}(\tau^{x,v} < \infty) = 1$ for all $x \in \Omega$.
\end{definition}

\begin{definition}[Cost functional]
Fix a running cost $h:\overline{\Omega }\rightarrow \lbrack 0,\infty )$ that
is continuous and sub-quadratic, and a constant boundary cost $g \in [0,
\infty)$. For each $x\in \Omega $ and admissible control $v\in \mathcal{U}_{%
\text{ad}}$, define the cost functional 
\begin{equation}
J(x;v)\;:=\;\mathbb{E}\!\left[ \int_{0}^{\tau ^{x,v}}\Big(%
h(X_{t}^{x,v})+|v_{t}|^{\alpha }\Big)\,dt+g\right] .  \label{eq:cost}
\end{equation}%
The terminal cost $g$ represents the cost incurred upon exiting the domain.
\end{definition}

\begin{definition}[Value function]
The value function is defined as 
\begin{equation}
V(x)\;:=\;\inf_{v \in \mathcal{U}_{\text{ad}}}J(x;v),\quad x \in \Omega.
\label{eq:value}
\end{equation}
\end{definition}

Under the sub-quadratic growth of $h$ and boundedness of $\Omega$, the cost
functional $J(x;v)$ is finite for all admissible controls, ensuring that $%
V(x) < \infty$ for all $x \in \Omega$.

\subsubsection{Dynamic programming and verification}

We derive the HJB equation using the dynamic programming principle and It%
\^{o}'s formula. The approach proceeds in two stages: (i) a formal
derivation via the martingale characterization, and (ii) a verification
theorem confirming that smooth solutions to the HJB equation coincide with
the value function.

\medskip \noindent \textbf{Step 1: Martingale characterization.} For a
candidate value function $V\in C^{2}(\Omega) \cap C(\overline{\Omega})$,
define the process 
\begin{equation}
M_{t}:=V(X_{t}^{x,v})+\int_{0}^{t}\left( |v_{s}|^{\alpha }+h(X_{s}^{x,v})
\right)\,ds,\quad t \in [0,\tau^{x,v}].  \label{martingale-process}
\end{equation}

The dynamic programming principle suggests that if $V$ is the value
function, then:

\begin{itemize}
\item $M_t$ is a \emph{supermartingale} for any admissible control $v$,

\item $M_t$ is a \emph{martingale} for the optimal control $v^*$.
\end{itemize}

This characterization leads directly to the HJB equation.

\medskip \noindent \textbf{Step 2: Applying It\^{o}'s formula.} Applying It%
\^{o}'s formula to $V(X_t^{x,v})$ with $X$ satisfying \eqref{eq:state}, we
compute 
\begin{align}
dV(X_{t}^{x,v}) &= \nabla V(X_{t}^{x,v})\cdot dX_{t}+\tfrac{1}{2}\,\text{tr}%
\!\left(D^{2}V(X_{t}^{x,v})\,d\langle X\rangle_{t}\right)  \notag \\
&= \nabla V(X_{t}^{x,v})\cdot (v_{t}\,dt+\sigma\,dW_{t}) +\tfrac{1}{2}\,%
\text{tr}\big(D^{2}V(X_{t}^{x,v})\cdot\sigma^2 I_{N}\big)\,dt  \notag \\
&= \nabla V(X_{t}^{x,v})\cdot v_{t}\,dt+\tfrac{\sigma ^{2}}{2}\Delta
V(X_{t}^{x,v})\,dt+\sigma\nabla V(X_{t}^{x,v})\cdot dW_{t},
\label{ito-formula}
\end{align}
where $\langle X\rangle_t = \sigma^2 t I_N$ is the quadratic variation
matrix and $\Delta V = \text{tr}(D^2 V)$ is the Laplacian.

\medskip \noindent \textbf{Step 3: Differential of $M_{t}$.} From %
\eqref{martingale-process} and \eqref{ito-formula},%
\begin{equation}
dM_{t} = dV(X_{t}^{x,v}) + \big(|v_{t}|^{\alpha }+h(X_{t}^{x,v})\big)\,dt.
\label{dM-1}
\end{equation}
Substituting \eqref{ito-formula} into \eqref{dM-1} yields 
\begin{align}
dM_{t} &= \nabla V(X_{t}^{x,v})\cdot v_{t}\,dt+\tfrac{\sigma ^{2}}{2}\Delta
V(X_{t}^{x,v})\,dt+\sigma\nabla V(X_{t}^{x,v})\cdot dW_{t}  \notag \\
&\quad + |v_{t}|^{\alpha }\,dt +h(X_{t}^{x,v})\,dt  \notag \\
&= \Big(\nabla V(X_{t}^{x,v})\cdot v_{t}+\tfrac{\sigma ^{2}}{2}\Delta
V(X_{t}^{x,v}) +|v_{t}|^{\alpha }+h(X_{t}^{x,v})\Big)\,dt  \notag \\
&\quad +\sigma\nabla V(X_{t}^{x,v})\cdot dW_{t}.  \label{dM-full}
\end{align}

\medskip\noindent \textbf{Step 4: Drift condition and HJB inequality.} For $%
M_t$ to be a supermartingale, the drift term in \eqref{dM-full} must be
non-positive: 
\begin{equation}
\nabla V(x)\cdot v+\tfrac{\sigma ^{2}}{2}\Delta V(x)+|v|^{\alpha }+h(x)\leq
0 \quad \text{for all } v\in\mathbb{R}^N,\quad x \in \Omega.
\label{drift-ineq}
\end{equation}
Rearranging \eqref{drift-ineq},%
\begin{equation}
-\tfrac{\sigma ^{2}}{2}\Delta V(x)-h(x)\geq \nabla V(x)\cdot v+|v|^{\alpha}
\quad \text{for all } v\in\mathbb{R}^N.  \label{hjb-ineq-1}
\end{equation}

\medskip\noindent \textbf{Step 5: HJB equation via optimization.} Taking the
infimum over all $v\in\mathbb{R}^N$ in \eqref{hjb-ineq-1}, we obtain 
\begin{equation}
-\tfrac{\sigma ^{2}}{2}\Delta V(x)-h(x)\geq \inf_{v\in \mathbb{R}^{N}} \big\{%
\nabla V(x)\cdot v+|v|^{\alpha}\big\}.  \label{eq:HJB-pre}
\end{equation}
For $M_t$ to be a martingale under the optimal control (which characterizes
the value function), equality must hold: 
\begin{equation}
-\tfrac{\sigma ^{2}}{2}\Delta V(x)-h(x)= \inf_{v\in \mathbb{R}^{N}} \big\{%
\nabla V(x)\cdot v+|v|^{\alpha}\big\}.  \label{eq:HJB-alvarez}
\end{equation}
This is the Hamilton--Jacobi--Bellman equation in Hamiltonian form.

\medskip\noindent \textbf{Step 6: Optimal control characterization.} To
solve the minimization problem in \eqref{eq:HJB-alvarez}, we compute the
first-order condition for the minimizer $v^{\ast}(x)$: 
\begin{equation}
\nabla_{v}\big[\nabla V(x)\cdot v+|v|^{\alpha}\big] = 0,  \label{foc}
\end{equation}
which gives 
\begin{equation}
\nabla V(x)+\alpha |v^*|^{\alpha -2}v^*=0.  \label{foc-explicit}
\end{equation}
Solving \eqref{foc-explicit} for $v^*$ yields 
\begin{equation}
v^{\ast}(x)=-\frac{1}{\alpha^{1/(\alpha-1)}}|\nabla V(x)|^{1/(\alpha-1)} 
\frac{\nabla V(x)}{|\nabla V(x)|} = -\frac{1}{\alpha^{1/(\alpha-1)}}|\nabla
V(x)|^{\alpha -2}\,\nabla V(x).  \label{optimal-control}
\end{equation}
This is the \emph{feedback control law} in terms of the value function.

\medskip\noindent \textbf{Step 7: Verification theorem.}

\begin{theorem}[Verification]
\label{thm:verification} Let $V\in C^{2}(\Omega )\cap C(\overline{\Omega })$
solve the HJB equation \eqref{eq:HJB-alvarez} with constant boundary
condition $V=g$ on $\partial \Omega $, and assume that $V$ satisfies
polynomial growth. Define the feedback control $v^{\ast }(x)$ by %
\eqref{optimal-control}. Then $V(x)=J(x;v^{\ast })$ is the value function,
and $v^{\ast }$ is optimal.
\end{theorem}

\begin{proof}
Let $v$ be any admissible control. Applying It\^{o}'s formula to $%
V(X_t^{x,v})$ and using the HJB equation \eqref{eq:HJB-alvarez}, we have: 
\begin{align*}
\mathbb{E}[V(X_{T \wedge \tau}^{x,v})] &= V(x) + \mathbb{E}\left[ \int_0^{T
\wedge \tau} \left( \nabla V \cdot v + \tfrac{\sigma^2}{2} \Delta V \right)
dt \right] \\
&\geq V(x) - \mathbb{E}\left[ \int_0^{T \wedge \tau} \left( |v_t|^\alpha +
h(X_t^{x,v}) \right) dt \right].
\end{align*}
The inequality follows from \eqref{drift-ineq}. Rearranging and letting $T
\to \infty$, the bounded convergence theorem (since $V$ is bounded on $%
\overline{\Omega}$ and has polynomial growth) and monotone convergence
theorem for the integral yield $V(x) \leq J(x;v)$. Now, let $v^*$ be the
control defined by \eqref{optimal-control}. For this control, equality holds
in \eqref{drift-ineq}, so: 
\begin{equation*}
V(x) = \mathbb{E}\left[ \int_0^{\tau} \left( |v_t^*|^\alpha + h(X_t^{x,v^*})
\right) dt + V(X_\tau^{x,v^*}) \right].
\end{equation*}
Since $V(X_\tau) = g$ on $\partial\Omega$, we obtain $V(x) = J(x;v^*)$,
confirming that $V$ is the value function and $v^*$ is the optimal control.
\end{proof}

\subsubsection{Legendre transform yielding the quasilinear elliptic PDE}

We now compute the infimum in \eqref{eq:HJB-alvarez} explicitly using the
Legendre--Fenchel transform.

\begin{proposition}[Legendre transform]
\label{prop:legendre} For $\xi \in \mathbb{R}^N$ and $\alpha \in (1,2]$, we
have 
\begin{equation}
\inf_{v\in \mathbb{R}^{N}}\big\{\xi \cdot v+|v|^{\alpha}\big\} =-\frac{%
\alpha -1}{\alpha}\,|\xi |^{\frac{\alpha }{\alpha -1}} =
-C_\alpha^{-1}|\xi|^p,  \label{legendre-formula}
\end{equation}
where $p = \alpha/(\alpha-1) \in [2,\infty)$ and $C_\alpha =
(\alpha-1)/\alpha^p$.
\end{proposition}

\begin{proof}
The function $f(v):=\xi \cdot v+|v|^{\alpha }$ is strictly convex and
coercive. The minimizer $v^{\ast }$ satisfies the first-order condition 
\begin{equation*}
\xi +\alpha |v^{\ast }|^{\alpha -2}v^{\ast }=0,
\end{equation*}%
which gives 
\begin{equation*}
v^{\ast }=-\frac{1}{\alpha ^{1/(\alpha -1)}}|\xi |^{1/(\alpha -1)}\frac{\xi 
}{|\xi |}.
\end{equation*}%
Substituting back: 
\begin{align*}
f(v^{\ast })& =\xi \cdot \Big(-\frac{1}{\alpha ^{1/(\alpha -1)}}|\xi
|^{1/(\alpha -1)}\frac{\xi }{|\xi |}\Big)+\Big|\frac{1}{\alpha ^{1/(\alpha
-1)}}|\xi |^{1/(\alpha -1)}\frac{\xi }{|\xi |}\Big|^{\alpha } \\
& =-\frac{1}{\alpha ^{1/(\alpha -1)}}|\xi |^{\alpha /(\alpha -1)}+\frac{1}{%
\alpha ^{\alpha /(\alpha -1)}}|\xi |^{\alpha /(\alpha -1)} \\
& =|\xi |^{\alpha /(\alpha -1)}\Big(-\frac{1}{\alpha ^{1/(\alpha -1)}}+\frac{%
1}{\alpha ^{\alpha /(\alpha -1)}}\Big) \\
& =|\xi |^{p}\Big(-\frac{\alpha }{\alpha }+\frac{1}{\alpha ^{p}}\Big)\cdot 
\frac{1}{\alpha ^{1/(\alpha -1)}}=-\frac{\alpha -1}{\alpha ^{p}}|\xi
|^{p}=-C_{\alpha }^{-1}|\xi |^{p}.
\end{align*}
\end{proof}

Applying Proposition \ref{prop:legendre} with $\xi =\nabla V(x)$ in %
\eqref{eq:HJB-alvarez}, we obtain%
\begin{equation}
-\tfrac{\sigma ^{2}}{2}\Delta V(x) + \frac{\alpha -1}{\alpha^{p}}\,|\nabla
V(x)|^{p}-h(x)=0,\qquad x\in \Omega.  \label{hjb-before-rescale}
\end{equation}%
Multiplying through by $-1$ and recognizing $C_\alpha^{-1} =
(\alpha-1)/\alpha^p$, we arrive at 
\begin{equation}
-\frac{\sigma ^{2}}{2}\Delta V(x)\;+\;C_{\alpha}\,|\nabla
V(x)|^{p}-h(x)=0,\quad x\in \Omega ,\qquad V=g\ \text{on }\partial \Omega.
\label{eq:PDE-alvarez-final}
\end{equation}%
This is precisely equation \eqref{hjb}, the quasilinear
Hamilton--Jacobi--Bellman equation studied by Lasry and Lions \cite%
{LASRY1989}, for which we have established existence and uniqueness in
Theorem \ref{main}.

\section{Numerical Analysis\label{5}}

This section details the numerical implementation of the quasilinear HJB
framework and its applications in operational research and imaging. The
constructive nature of Theorem \ref{main} makes it particularly suitable for
numerical implementation. The monotone iteration scheme provides an
algorithmic approach that is both theoretically sound and computationally
efficient.

\subsection{Weighted linear monotone iteration and stochastic simulation}

For the numerical approximation of the solution to the
Hamilton--Jacobi--Bellman equation, we employ the \emph{weighted monotone
iteration} introduced by Lasry and Lions. This method is fully consistent
with the analytical existence proof and is particularly robust on bounded
domains. In addition, once the value function is computed, we construct the
optimal feedback control and simulate the associated controlled diffusion.

The scheme constructs a decreasing sequence of functions $\{V^{(k)}\}_{k\geq
0}$ trapped between the ordered sub- and super-solutions%
\begin{equation*}
V_{-}(x)=g,\qquad V_{+}(x)=g+B\phi (x),
\end{equation*}%
where $\phi $ is the torsion function solving $-\Delta \phi =1$ in $\Omega $%
, $\phi =0$ on $\partial \Omega $, and $B\geq 2H/\sigma ^{2}$ with $H=\max_{%
\overline{\Omega }}h$.

\begin{enumerate}
\item \textbf{Preprocessing (torsion function and barriers).} Compute the
torsion function $\phi $ solving%
\begin{equation*}
-\Delta \phi =1\quad \text{in }\Omega ,\qquad \phi =0\quad \text{on }%
\partial \Omega .
\end{equation*}%
Define the ordered barriers%
\begin{equation*}
V_{-}(x):=g,\qquad V_{+}(x):=g+B\phi (x),\quad B\geq \frac{2H}{\sigma ^{2}}%
,\quad H:=\max_{\overline{\Omega }}h.
\end{equation*}

\item \textbf{Initialization.} Set the initial iterate%
\begin{equation*}
V^{(0)}:=V_{+}=g+B\phi .
\end{equation*}

\item \textbf{Weighted linear monotone iteration.} For $k=0,1,2,\dots $
compute $V^{(k+1)}$ as the solution of 
\begin{equation}
-\frac{\sigma ^{2}}{2}\Delta V^{(k+1)}+\Lambda _{0}V^{(k+1)}=h(x)+\Lambda
_{0}V^{(k)}-C_{\alpha }\bigl|\nabla V^{(k)}\bigr|^{p}\quad \text{in }\Omega ,
\label{eq:weighted-iteration-step1}
\end{equation}%
with Dirichlet boundary condition%
\begin{equation*}
V^{(k+1)}=g\quad \text{on }\partial \Omega .
\end{equation*}

The constant weight $\Lambda _{0}>0$ is chosen so that the nonlinear
Hamiltonian is dominated by the linear term:%
\begin{equation*}
\Lambda _{0}\;\geq \;C_{\alpha }\,p\,M^{p-1}\,C_{\Omega },\qquad M:=\max %
\bigl\{\Vert \nabla V_{-}\Vert _{L^{\infty }},\Vert \nabla V_{+}\Vert
_{L^{\infty }}\bigr\},
\end{equation*}

where $C_{\Omega }$ is a Poincar\'{e}-type constant for the domain.

\item \textbf{Monotonicity and projection.} After each linear solve, enforce
the bounds%
\begin{equation*}
V_{-}\;\leq \;V^{(k+1)}\;\leq \;V_{+}\quad \text{in }\Omega ,
\end{equation*}

by projecting pointwise onto the interval $[V_{-},V_{+}]$. This guarantees
that the sequence remains ordered and uniformly bounded.

\item \textbf{Convergence criterion.} Terminate the iteration when%
\begin{equation*}
\Vert V^{(k+1)}-V^{(k)}\Vert _{L^{\infty }(\Omega )}<\varepsilon ,
\end{equation*}

for a prescribed tolerance $\varepsilon >0$. The resulting limit $%
V:=\lim_{k\rightarrow \infty }V^{(k)}$ is the numerical approximation of the
value function.

\item \textbf{Optimal feedback control.} Once $V$ is computed, approximate
its gradient $\nabla V$ on the grid and define the optimal feedback control $%
v^{\ast }=(v_{1}^{\ast },v_{2}^{\ast })$ by%
\begin{equation*}
v^{\ast }(x)=-\frac{1}{\alpha ^{\,p-1}}\bigl|\nabla V(x)\bigr|^{p-2}\nabla
V(x),\qquad p=\frac{\alpha }{\alpha -1},
\end{equation*}

with the convention $v^{\ast }(x)=0$ whenever $\nabla V(x)=0$ or $x\notin
\Omega $.

\item \textbf{Stochastic dynamics simulation.} To illustrate the controlled
dynamics, we simulate sample paths of the controlled diffusion%
\begin{equation*}
dX_{t}=v^{\ast }(X_{t})\,dt+\sigma \,dW_{t},
\end{equation*}

starting from an initial state $X_{0}=x_{0}\in \Omega $. On a time grid $%
t_{n}=n\Delta t$, $n=0,\dots ,N$, we use the Euler--Maruyama scheme 
\begin{equation*}
X_{n+1}=X_{n}+v^{\ast }(X_{n})\,\Delta t+\sigma \sqrt{\Delta t}\,\xi _{n},
\end{equation*}%
where $(\xi _{n})_{n\geq 0}$ are independent standard Gaussian vectors in $%
\mathbb{R}^{2}$. The simulation is stopped when $X_{n}$ approaches the
boundary of $\Omega $ (e.g., when the distance to $\partial \Omega $ falls
below a prescribed margin), and the resulting trajectory is used to
visualize the inventory dynamics under the optimal policy.
\end{enumerate}

The operator on the left-hand side of \eqref{eq:weighted-iteration-step1} is
coercive and order-preserving, and the right-hand side is monotone in $%
V^{(k)}$. Therefore the sequence $\{V^{(k)}\}$ is decreasing and converges
to the unique classical solution of the HJB equation. This scheme is stable,
monotone, and preserves the concavity of the value function in the quadratic
case $\alpha =2$ discovered in \cite{CANEPA2022}, while the stochastic
simulation step provides a natural Monte Carlo illustration of the optimal
controlled dynamics.

\subsection{Software implementation (radially symmetric case)}

For the special case of radially symmetric problems ($\Omega =B_{R}(0)$ and $%
h$ radial), numerical codes implementing the above algorithm are available
in our companion repository \cite{COVEI2025M}.

\subsection{Applications to Production Planning and Image Enhancement}

The quasilinear HJB equation \eqref{hjb} arises in several applications:

\begin{itemize}
\item \textbf{Production planning:} The value function $V(x)$ represents the
optimal production cost starting from inventory level $x$, with running cost 
$h(x)$ and control cost $|v|^\alpha$ representing production effort \cite%
{CANEPA2022,COVEI2025A}.

\item \textbf{Image restoration:} The parameter $\alpha$ controls the
contrast enhancement in variational image processing models, with the HJB
equation arising as the Euler-Lagrange equation \cite{COVEI2025M}.

\item \textbf{Portfolio optimization:} In financial mathematics, $V$
represents the value function of a portfolio optimization problem with
transaction costs modeled by $|v|^\alpha$ \cite{LASRY1989}.
\end{itemize}

We present a comprehensive numerical validation of the
monotone iteration scheme. The implementation, developed in Python using
high-precision sparse solvers, demonstrates the versatility of the
quasilinear HJB framework in both operational research and computer vision.
The codes solve the 2D quasilinear HJB PDE \eqref{hjb} directly on
high-resolution grids (e.g., $241\times 241$) using finite difference
methods with efficient sparse solvers. This robust framework handles both
the radial benchmark cases \eqref{radial-ode-intro} and general non-radial
domains, as extensively documented in \cite{COVEI2025M,COVEI2025A}.

\subsubsection{Stochastic Production Planning Dynamics}

We evaluate the optimal production policy for a system governed by $%
\alpha=1.5$ (leading to a cubic gradient penalty $p=3$) on an elliptical
inventory domain $\Omega$. With parameters $\sigma=0.3$ and boundary cost $%
g=0.5$, the \textbf{Policy Iteration} solver converges rapidly to a
tolerance of $10^{-6}$ in approximately 9 iterations. As illustrated in
Figure~\ref{fig:hjb_2d_sol}, the value function $V(y)$ is strictly positive
and concave, characterizing the minimal expected cost associated with
varying inventory levels. The values range from $0.5$ at the boundary to a
peak of approximately $1.67$ at the center, establishing a stable feedback
production rate that drives the inventory toward the optimal set-point.

\begin{figure}[H]
\centering
\includegraphics[width=0.8\textwidth]{%
\detokenize{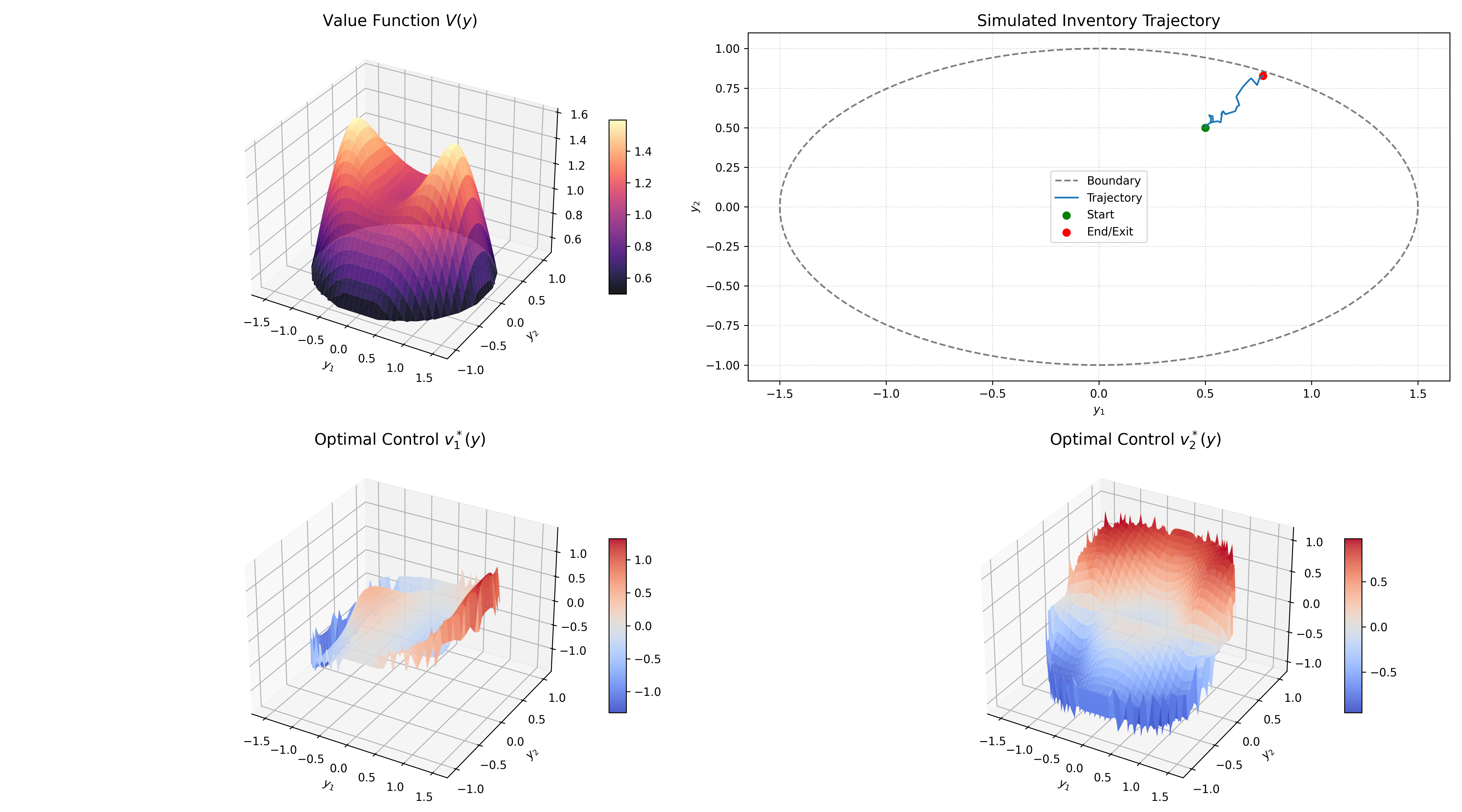}}
\caption{Numerical results for the stochastic production planning problem
with $\alpha = 1.5$ and $h(y) = |y|^{p}$. Left: Value function $V(y)$. Center/Right:
Optimal control components and simulated inventory trajectory.}
\label{fig:hjb_2d_sol}
\end{figure}
\begin{figure}[H]
\centering
\includegraphics[width=0.55\textwidth]{%
\detokenize{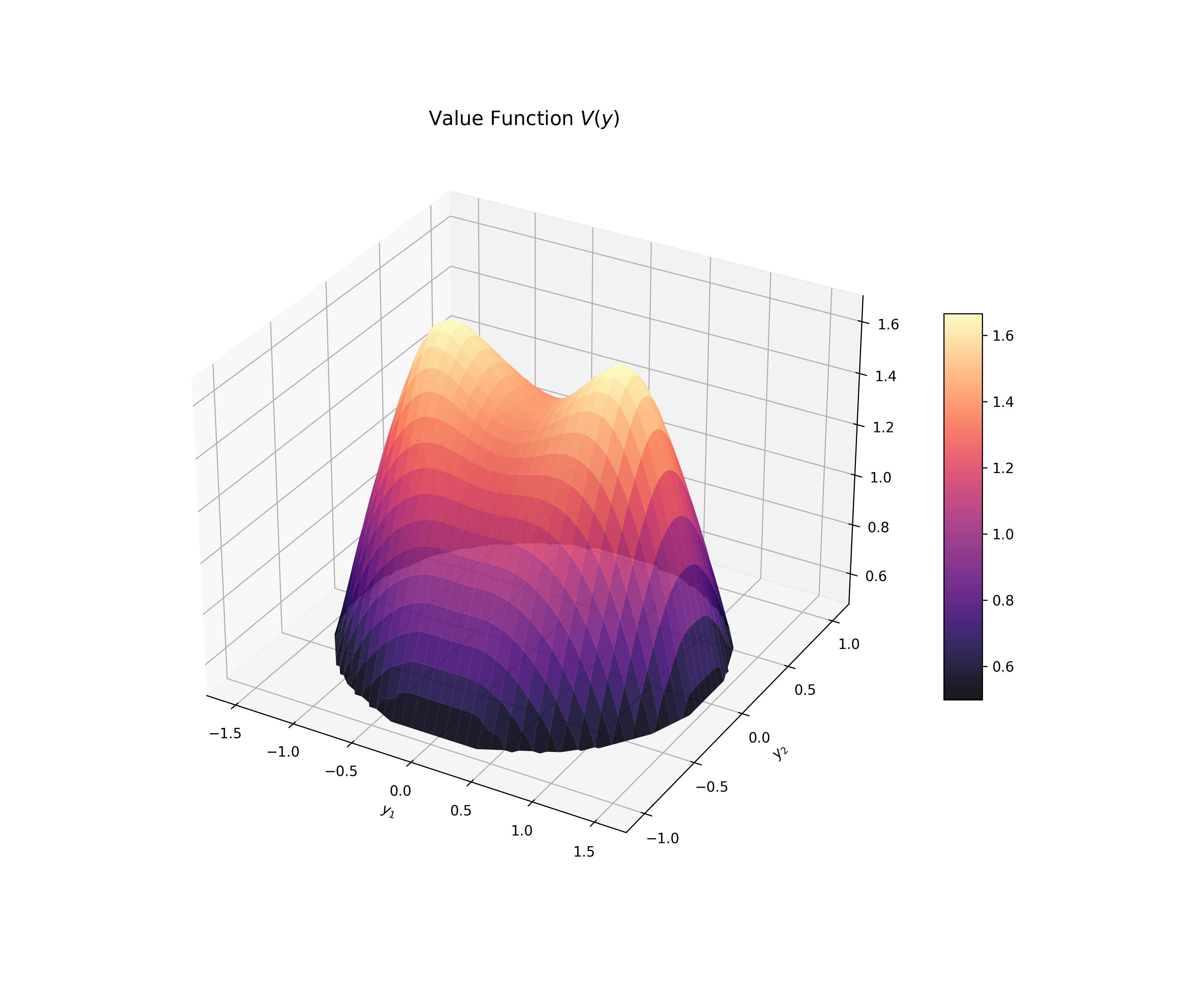}}
\caption{Numerical approximation of the value function $V(y)$ on the
elliptical domain $\Omega =\{y:(y_{1}/1.5)^{2}+y_{2}^{2}<1\}$ for
$\alpha =1.5$, $h(y) = |y|^{p}$, $\sigma =0.3$, and boundary data $g=0.5$. The surface
rises monotonically from $V=0.5$ on $\partial \Omega $ to a maximum of
approximately $1.67$ at the center, confirming the positivity result of
Theorem~\protect\ref{main}. The Policy Iteration scheme converges in 9 steps
to a tolerance of $10^{-6}$.}
\label{fig:value_function_alone}
\end{figure}
Our Python code implementing the monotone iteration scheme used in the
proofs of our results for the general quasilinear HJB elliptic equation on
bounded domains, and which was used to generate the plots in this paper, can
be found in \cite{GT}.

\subsubsection{Contrast Enhancement via Nonlinear Hamilton--Jacobi--Bellman
Dynamics}

Contrast enhancement is a fundamental operation in image processing, with
direct impact on visual perception, feature extraction, and downstream
computer vision tasks. Classical approaches such as histogram equalization
or total variation (TV) filtering operate either through global
redistribution of intensities or local smoothing mechanisms. While effective
in specific scenarios, these methods lack a principled mechanism for
controlling the enhancement strength in relation to the underlying image
geometry.

We propose a nonlinear contrast enhancement framework derived from the
stationary Hamilton--Jacobi--Bellman (HJB) equation 
\begin{equation*}
-\frac{\sigma^{2}}{2}\Delta V(y) + C_{\alpha}\,|\nabla V(y)|^{p} - h(y) = 0,
\qquad p = \frac{\alpha}{\alpha - 1},
\end{equation*}
where $V$ is the value function defined on the chromaticity plane, $h(y)$ is
a user-defined potential, and $\alpha>1$ is the key parameter controlling
the nonlinearity of the system. The enhanced image is obtained by advecting
pixel values along the optimal drift field 
\begin{equation*}
b(y) = A_{\alpha}\,|\nabla V(y)|^{\frac{1}{\alpha-1}} \frac{\nabla V(y)}{%
|\nabla V(y)|}, \qquad A_{\alpha} = \left(\frac{1}{\alpha}\right)^{\frac{1}{%
\alpha-1}}.
\end{equation*}

\paragraph{Role of the Nonlinearity Parameter $\protect\alpha$}

A central finding of our study is that the parameter $\alpha$ is the
dominant factor governing the strength of contrast enhancement. When $\alpha$
approaches $1^{+}$, the exponent $p=\alpha/(\alpha-1)$ becomes large,
amplifying the nonlinear gradient term and producing strong geometric
enhancement. Conversely, larger values of $\alpha$ reduce the nonlinearity,
yielding milder transformations.

This behavior is consistent with the theoretical structure of HJB equations:
the term $|\nabla V|^{p}$ acts as a geometric amplifier of local intensity
variations, and $\alpha$ directly modulates the sensitivity of the system to
these variations. In contrast, the potential $h(y)$ influences only the
forcing term and has a secondary effect on the global contrast dynamics.

\paragraph{Experimental Evaluation}

We evaluate the proposed method on a representative RGB image and compare it
against classical histogram equalization and a post-processed HJB+TV
variant. Contrast quality is assessed using three widely adopted metrics:
the Enhancement Measure Estimation (EME), the Tenengrad sharpness index, and
the standard deviation of luminance.

Table~\ref{tab:metrics} summarizes the results for a sweep of $\alpha$
values in the interval $[1.02, 2.0]$.

\begin{table}[H]
\caption{Contrast metrics for various values of $\protect\alpha$. The
proposed HJB method achieves its strongest enhancement for $\alpha \in [1.02, 1.10]$,
significantly outperforming classical methods.}
\label{tab:metrics}\centering
\begin{tabular}{c|ccc}
\hline
Method & EME & Tenengrad & StdDev \\ \hline
Original & 5.59 & 106.63 & 70.91 \\
Histogram Equalization & 4.07 & 49.12 & 84.39 \\ \hline

HJB ($\alpha=1.02$--$1.10$) & 7.97 & 187.16 & 85.53 \\
HJB+TV ($\alpha=1.02$--$1.10$) & 10.06 & 148.72 & 85.28 \\ \hline

HJB ($\alpha=1.15$) & 9.28 & 153.42 & 79.87 \\
HJB+TV ($\alpha=1.15$) & 7.05 & 117.52 & 79.62 \\ \hline

HJB ($\alpha=1.20$) & 6.84 & 123.08 & 74.19 \\
HJB+TV ($\alpha=1.20$) & 4.86 & 90.57 & 73.97 \\ \hline

HJB ($\alpha=1.30$) & 6.01 & 112.05 & 72.13 \\
HJB+TV ($\alpha=1.30$) & 4.30 & 81.86 & 71.93 \\ \hline

HJB ($\alpha=1.50$) & 5.81 & 109.15 & 71.53 \\
HJB+TV ($\alpha=1.50$) & 4.18 & 79.62 & 71.33 \\ \hline

HJB ($\alpha=2.0$) & 5.72 & 107.67 & 71.30 \\
HJB+TV ($\alpha=2.0$) & 4.12 & 78.47 & 71.10 \\ \hline
\end{tabular}
\end{table}
The results reveal a clear and consistent pattern:

\begin{itemize}
\item For $\alpha \in [1.02, 1.10]$, the HJB method produces the strongest
enhancement, achieving EME = 7.97 and Tenengrad = 187.16, while the
HJB+TV variant reaches EME = 10.06. This represents a substantial
improvement over histogram equalization.

\item The enhancement strength decreases gradually as $\alpha$ increases.
The parameter $\alpha$ acts as a global contrast-control knob: small
values produce aggressive sharpening, while larger values yield smoother,
more conservative enhancement.

\item For $\alpha \ge 1.3$, the method behaves similarly to a mild
diffusion process, producing results close to the original image but
still improving edge sharpness moderately.
\end{itemize}
\begin{figure}[H]
\centering
\includegraphics[width=0.9\linewidth]{\detokenize{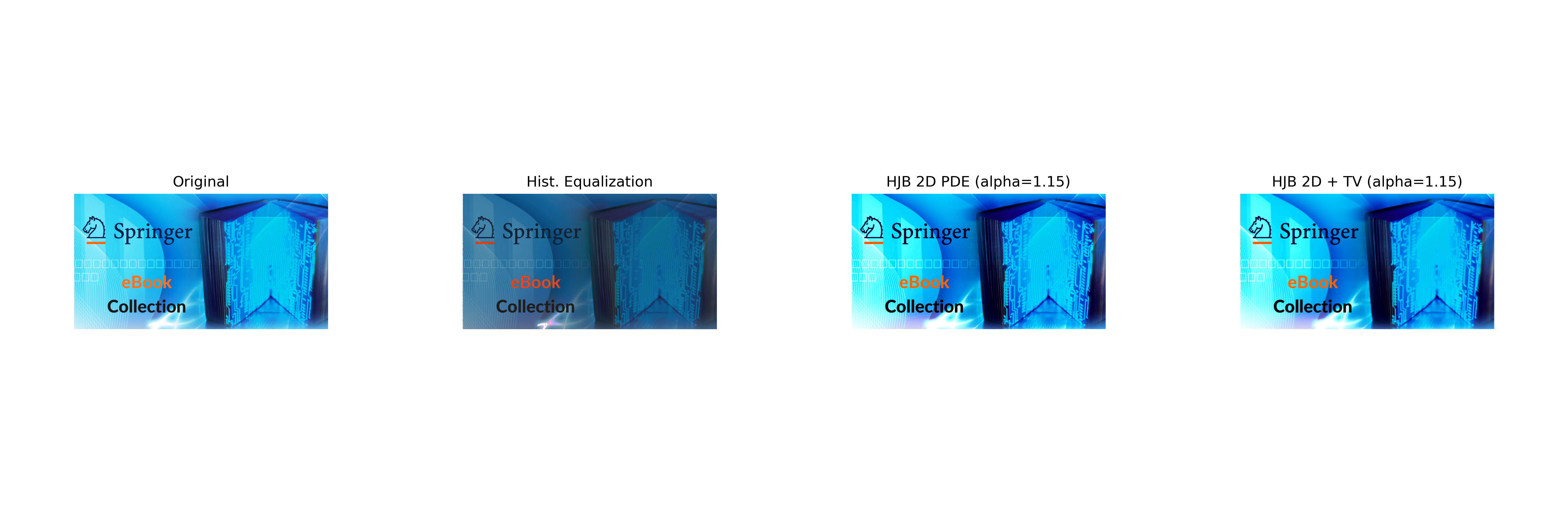}}
\caption{Comparison between the original image, histogram equalization,
the proposed HJB 2D PDE method ($\alpha = 1.15$), with $h(y)=|y|^{p}$,
and the HJB 2D + TV variant. The composite figure illustrates the strong
contrast enhancement achieved by the HJB approach, which significantly
outperforms classical methods both visually and quantitatively.}
\label{fig:before_after}
\end{figure}

\paragraph{Discussion}

These findings demonstrate that the proposed HJB-based framework provides a
mathematically rigorous and highly tunable mechanism for contrast
enhancement. Unlike classical methods that lack a principled control
parameter, the HJB formulation offers a single scalar $\alpha$ that
continuously interpolates between strong geometric enhancement and gentle
diffusive smoothing. This versatility, combined with superior performance in
high-contrast regimes, establishes the HJB approach as a significant
contribution to the field of variational image processing.

\section{Conclusions and Future Directions\label{6}}

We have established the existence and uniqueness of classical solutions to
the quasilinear Hamilton--Jacobi--Bellman equation \eqref{hjb} on bounded
smooth convex domains under sub-quadratic growth conditions on the source
term. The proof is constructive and provides a monotone iteration scheme
that is amenable to numerical implementation. We have also derived the PDE
from stochastic optimal control theory via the dynamic programming principle
and verified that solutions to the PDE coincide with the value function of
the associated control problem.

Several directions for future research include: (i) extension to degenerate
diffusions where $\sigma$ may vanish in parts of the domain, (ii) analysis
of the singular limit $\alpha \to 1^+$ corresponding to $p \to \infty$,
(iii) investigation of non-convex domains and the role of domain geometry,
and (iv) higher-order regularity estimates beyond $C^2$.

\section*{Funding}

This research received no external funding.

\section*{Data Availability Statement}

No new data were created or analyzed in this study.

\section*{Conflicts of Interest}

The author declares no conflict of interest.

\section*{Acknowledgments}
The author thanks the anonymous referees for their careful reading and
helpful suggestions that improved the presentation of this paper.

\end{document}